\documentclass[11pt,a4paper]{amsart}

\usepackage{amsthm, amsfonts, amssymb, amsmath, latexsym, enumerate,array}
\usepackage{ulem}
\usepackage{amscd} 
\usepackage[all]{xy}
\usepackage[utf8]{inputenc}
\usepackage{hyperref,cite,mdwlist,mathrsfs}
\usepackage{tikz}
\usepackage{pgf,tikz}
\usetikzlibrary{arrows}

\usepackage{anysize}
\marginsize{3.0cm}{3.0cm}{3.3cm}{3.3cm}
\usepackage{setspace}

\newtheorem{thm}{Theorem}[section]
\newtheorem{lem}[thm]{Lemma}

\newtheorem{prop}[thm]{Proposition}

\newtheorem*{thm*}{Theorem}
\newtheorem*{cnj*}{Conjecture}

\theoremstyle{definition}
\newtheorem{rmk}[thm]{Remark}

\newtheorem{dfn}[thm]{Definition}

\newtheorem*{conj*}{Conjecture}

\newcommand{\sO}{\mathscr{O}}

\newcommand{\cT}{\mathcal{T}}

\DeclareMathOperator{\HH}{H}

\newcommand{\Z}{\mathbb Z}

\newcommand{\N}{\mathbb N}
\newcommand{\p}{\mathbb P}

\newcommand{\kk}{{\boldsymbol{k}}}

\allowdisplaybreaks[1]
\begin{document}



\title{
Free divisors in a pencil of curves}

\author{Jean Vallès}
\email{{\tt jean.valles@univ-pau.fr}}
\address{Université de Pau et des Pays de l'Adour \\
  Avenue de l'Université - BP 576 - 64012 PAU Cedex - France}
\urladdr{\url{http://jvalles.perso.univ-pau.fr/}}

\keywords{Arrangements of curves, Pencil of curves, Freeness of arrangements, Logarithmic sheaves}
\subjclass[2010]{14C21, 14N20, 32S22, 14H50}


\thanks{Author partially supported by ANR GEOLMI ANR-11-BS03-0011 and PHC-SAKURA 31944VE}


\begin{abstract} A plane curve $D\subset \p^2(\kk)$ where $\kk$ is a field of characteristic zero is free if 
its associated  sheaf  $\cT_D$  of vector fields tangent to  $D$
is a free $\sO_{\p^2(\kk)}$-module (see \cite{S} or \cite{OT} for a definition in a more general context). Relatively few 
free curves are known. Here we prove that a divisor $D$ consisting of a union of curves of a pencil of plane projective curves with the same degree
 and with a smooth base locus is 
a free divisor if and only if $D$ contains all the singular members of the pencil and its Jacobian ideal is locally a complete intersection.

\end{abstract}

\maketitle

\section{Introduction}

Let $\kk$ be a field of characteristic zero and let  $S=\kk[x,y,z]$ be the graded ring such that  $\p^2=\mathrm{Proj}(S)$. We write  $\partial_x:=\frac{\partial}{\partial x}$,  
 $\partial_y:=\frac{\partial}{\partial y}$,  $\partial_z:=\frac{\partial}{\partial z}$ and $\nabla F =(\partial_x F,\partial_y F,\partial_z F)$ for a homogenous polynomial $F\in S$.
  
Let $D=\{F=0\}$ be a  reduced curve of degree $n$. The kernel $\cT_D$  of the map $\nabla F$ is  a rank two reflexive sheaf, hence a vector bundle on $\p^2$. It is 
 the rank two vector bundle of vector fields tangent along $D$, defined by the
following exact sequence:
$$ 
\begin{CD} 0 @>>> \cT_{D} @>>> \sO_{\p^2}^{3} @>\nabla
 F>> \mathcal{J}_{\nabla F}(n-1)@>>>0,
\end{CD}
$$
where the sheaf $\mathcal{J}_{\nabla F}$ (also denoted $\mathcal{J}_{\nabla D}$ in this text)  is the Jacobian ideal  of $F$.
Set theoretically $\mathcal{J}_{\nabla F}$ defines the singular points of the divisor $D$. For instance if $D$  consists of $s$ generic lines 
then  $\mathcal{J}_{\nabla D}$ defines the set of $\binom{s}{2}$ vertices of $D$.

\begin{rmk}
\label{section-derivation}
A non zero section $s\in \HH^0(\cT_D(a))$, for some shift $a\in \N$, corresponds to a derivation $\delta=P_a\partial_x+Q_a\partial_y+R_a\partial_z$ verifying $\delta(F)=0$, 
where $(P_a, Q_a, R_a)\in \HH^0(\sO_{\p^2}(a))^3$.
\end{rmk}
In some particular cases that can be found in \cite{OT}, $\cT_D$ is a free $\sO_{\p^2}$-module; it means that there are two vector fields of degrees $a$ and $b$ that form a basis
of $\bigoplus_n \HH^0(\cT_D(n))$ ($D$ is said to be free with exponents $(a,b)$); it arises, for instance, when $D$ is the union of the nine  inflection lines of a smooth cubic curve.
The notion of free divisor was introduced by Saito \cite{S} for reduced divisors and studied by Terao \cite{T} for hyperplane arrangements.
Here we  recall a definition of freeness for  projective curves. For a more general definition we refer to 
Saito \cite{S}. 
\begin{dfn}
A reduced curve $D\subset \p^2$ is \textit{free with exponents} $(a,b)\in \N^2$ if  $$\cT_{D}\simeq \sO_{\p^2}(-a)\oplus \sO_{\p^2}(-b).$$
\end{dfn}
A smooth curve of degree $\ge 2$ is not free, an irreducible curve of degree $\ge 3$ with only nodes and cusps as singularities is not free
(see \cite[Example 4.5]{DS}). Actually few examples of free curves are known and of course  very few families of free curves are known. One such family 
can be found in \cite[Prop. 2.2]{ST}.   
 
 \smallskip
 
In a personal communication it was conjectured by E. Artal and J.I. Cogolludo that the union of all the singular members of a pencil of plane curves (assuming that the general one is smooth) should be free.
Three different cases occur:
\begin{itemize}
 \item the base locus is smooth (for instance the union of six lines in a pencil of conics passing through four distinct points);
 \item the base locus is not smooth but every curve in the pencil is reduced (for instance the four lines in a pencil of conics where two of the base points are infinitely near points);
 \item the base locus is not smooth and there exists exactly one non reduced curve in the pencil (for instance three lines in a pencil of bitangent conics).
\end{itemize}
In the third case 
the divisor of singular members is not reduced but  its reduced structure is expected to be free. 

We point out that if two distinct curves of the pencil are not reduced then all curves are singular. Even in this case, we believe that a free divisor can be obtained by chosing 
a finite number of reduced components through all the singular points. 

In this paper we prove (see theorem  \ref{thm1}) that a divisor $D$ consisting of a union of curves
 of a pencil of degree $n$ plane curves with a smooth  base locus (i.e. the base locus consists of $n^2$ distinct points)
is a free divisor with determined exponents if and only if $D$ contains all the singular members of the pencil and its Jacobian ideal is
locally a complete intersection (this is always the case when $D$ is a union of lines or when the singular points of $D$ are all ordinary double points).
More generally,  we describe the vector bundle of logarithmic vector fields tangent to any union of curves of the pencil (see theorem \ref{thm0})
by studying one particular vector field  ``canonically tangent'' to the pencil, 
that is introduced in the key lemma \ref{lemcle}.

This gives already a new and easy method to produce 
free divisors. 

\smallskip

I thank J. I. Cogolludo for his useful comments.
\section{Pencil of plane curves}
\label{sect2}
\subsection{Generalities and notations}
Let $\{f=0\}$ and $\{g=0\}$ be two reduced curves of degree $n\ge 1$ with no common component. For any  $(\alpha, \beta)\in \p^1$  the curve  $C_{\alpha,\beta}$ is
 defined by the equation 
$\{\alpha f+\beta g=0\}$ and  $\mathcal{C}(f,g)=\{C_{\alpha,\beta}| (\alpha, \beta)\in \p^1\}$ is the pencil of all these curves.

\smallskip

In section \ref{sect2}  we will assume that the general member of the pencil $\mathcal{C}(f,g)$ is a smooth curve and that $C_{\alpha, \beta}$ is reduced for every $(\alpha, \beta)\in \p^1$. 

\smallskip
 
Under these assumptions    there are finitely many singular curves in $\mathcal{C}(f,g)$ but also finitely many singular points.
We recall that the degree of the discriminant variety of degree $n$ curves is $3(n-1)^2$ (it is a particular case of the Boole formula; see \cite[Example 6.4]{Tev}). Since the general curve in the pencil is smooth, the line defined by the pencil $\mathcal{C}(f,g)$ 
in the space of degree $n$ curves meets the discriminant variety along a finite scheme of length $3(n-1)^2$ (not empty for $n\ge 2$).
The number of singular points is  of course related to the multiplicity of the singular curves in the pencil as we will see below. 

\smallskip

Let us fix some notation. The scheme defined by the ideal sheaf $\mathcal{J}_{\nabla C_{\alpha_i,\beta_i}}$ is denoted by $Z_{\alpha_i,\beta_i}$. 
The union of all the singular members of the pencil $\mathcal{C}(f,g)$ form a divisor $D^{\mathrm{sg}}$.
A union of $k\ge 2$ distinct members of $\mathcal{C}(f,g)$ is denoted by $D_k$. The whole set of singularities of the pencil is denoted by $\mathrm{Sing}(\mathcal{C})$.

\subsection{Derivation tangent to a smooth pencil}

Let us consider the following derivation, associated ``canonically'' to the pencil:
\begin{lem}
\label{lemcle}
For any union $D_k$ of $k\ge 1$ members of the pencil there exists a non zero 
 section $s_{\delta,k}\in \HH^0(\cT_{D_k}(2n-2))$ induced by the derivation
 $$ \delta= (\nabla f \wedge \nabla g).\nabla=(\partial_y f\partial_z g -\partial_z f\partial_y g)\partial_x+ (\partial_z f\partial_x g -\partial_x f\partial_z g)\partial_y+
(\partial_x f\partial_y g -\partial_y f\partial_x g)\partial_z.$$
\end{lem}
\begin{proof}
 Since $\delta(\alpha f+\beta g)=\mathrm{det}(\nabla f, \nabla g, \nabla (\alpha f+\beta g))=0$ we have for any $k\ge 1$,  
$$\delta(f)=\delta(g)=\delta(\alpha f+\beta g)=\delta(\prod_{i=1}^k(\alpha_i f+\beta_i g))=0.$$
According to the remark \ref{section-derivation} it gives the desired section.
\end{proof}
Let us introduce a rank two sheaf $\mathcal{F}$ defined by the following exact sequence:
$$
\begin{CD}
 0 @>>> \sO_{\p^2}(2-2n) @>\nabla f\wedge \nabla g>> \sO^3_{\p^2} @>>> \mathcal{F}@>>> 0.
\end{CD}
$$
If we denote by  $\mathrm{sg}(\mathcal{F})$ the singular scheme of $\mathcal{F}$ supported by   the set $\{p\in \p^2 | \mathrm{rank}(\mathcal{F}\otimes \sO_p)>2\}$ of singular points of $\mathcal{F}$, we have: 
\begin{lem}\label{singF} 
A point $p\in \p^2$ belongs to $\mathrm{sg}(\mathcal{F})$ if and only if two smooth members of the pencil share the same tangent line at $p$ or one curve of the pencil is singular at $p$.
 Moreover  $\mathrm{sg}(\mathcal{F})$ is a finite closed scheme with length  $l(\mathrm{sg}(\mathcal{F}))=3(n-1)^2$.
\end{lem}
\begin{rmk}
 Let us precise that if two smooth members intersect then all the smooth members of the pencil  intersect with the same tangency.  
\end{rmk}

\begin{rmk}
\label{smoothbaselocus}
 If the base locus of $\mathcal{C}(f,g)$ consists of $n^2$ distinct points then two curves of the pencil meet transversaly at the base points and 
 $p\in \mathrm{sg}(\mathcal{F})$ if and only if $p$ is a singular point for a unique curve $C_{\alpha, \beta}$ in the pencil and does not belong to the base locus.
 One can assume that $p\in \mathrm{sg}(\mathcal{F})$ is  singular for $\{f=0\}$.  In other words, when the base locus is smooth, the support of  
 $ \mathrm{sg}(\mathcal{F})$ is  $\mathrm{Sing}(\mathcal{C})$.
\end{rmk}
\begin{proof}
The support of the $\mathrm{sg}(\mathcal{F})$ is also defined by $\{p\in \p^2 | (\nabla f\wedge \nabla g)(p)=0\}$. The
zero scheme defined by $\nabla f\wedge \nabla g$
and  $\nabla (\alpha f+\beta g)\wedge \nabla g$ are clearly the same;  it means  that the singular points of any member in the pencil is a singular point for 
$\mathcal{F}$.
One can also obtain $(\nabla f\wedge \nabla g)(p)=0$ at a smooth point when the vectors $(\nabla f)(p)$ and  $(\nabla g)(p)$ are proportional i.e. when 
two smooth curves of the pencil share the same tangent line at $p$.

Since every curve of the pencil is reduced  $\mathrm{sg}(\mathcal{F})$  is finite, its length can be computed by 
writing the resolution of the ideal $\mathcal{J}_{\mathrm{sg}(\mathcal{F})}$ (for a sheaf of ideal $\mathcal{J}_Z$  defining a finite scheme $Z$ of length $l(Z)$,
we have $c_2(\mathcal{J}_Z)=l(Z)$). Indeed, if we dualize the following exact sequence
$$ 
\begin{CD}
 0@>>> \sO_{\p^2}(2-2n)  @>\nabla f \wedge \nabla g>>\sO_{\p^2}^3  @>>> \mathcal{F} @>>>0
\end{CD}
$$
we find, according to Hilbert-Burch theorem,
$$ 
\begin{CD}
 0@>>> \mathcal{F}^{\vee}  @>(\nabla f, \nabla g)>>\sO_{\p^2}^3  @>\nabla f \wedge \nabla g>>  \sO_{\p^2}(2n-2). 
\end{CD}
$$
It proves that $\mathcal{F}^{\vee}=\sO_{\p^2}(1-n)^2$ and that the image of the last map is $\mathcal{J}_{\mathrm{sg}(\mathcal{F})}(2n-2)$.

Then $l(\mathrm{sg}(\mathcal{F}))=3(n-1)^2$. We point out that this number is the degree of the discriminant variety of degree $n$ curves.
\end{proof}

Now let us call $D_k$ the divisor defined by $k\ge 2$ members of the pencil and let us consider the section 
$s_{\delta,k}\in \HH^0(\cT_{D_k}(2n-2))$ corresponding (see remark \ref{section-derivation}) to the derivation $\delta$.  Let $Z_k:=Z(s_{\delta,k})$ be the zero locus  of $s_{\delta,k}$.
\begin{lem}
\label{sectiondelta}
 The section $s_{\delta,k}$ vanishes in codimension at least two.
\end{lem}
\begin{proof}
Let us consider the following commutative diagram 

$$ 
\begin{CD}
@.   @. 0 @. 0 @.\\
@.   @. @VVV  @VVV\\
0@>>> \sO_{\p^2}(2-2n)  @>s_{\delta,k}>>\cT_{D_k} @>>> \mathcal{Q} @>>>0\\
@. @| @VVV @VVV   @.\\
 0@>>> \sO_{\p^2}(2-2n)  @>\nabla f \wedge \nabla g>>\sO_{\p^2}^3  @>>> \mathcal{F} @>>>0\\
@.   @. @VVV  @VVV\\
@.   @. \mathcal{J}_{\nabla D_k}(nk-1)@=  \mathcal{J}_{\nabla D_k}(nk-1) @.\\
@.   @. @VVV  @VVV\\
@.   @. 0 @. 0 @.
\end{CD}
$$
where $\mathcal{Q}=\mathrm{coker}(s_{\delta,k})$. Assume that $Z_k$ contains
a divisor $H$. Tensor now the last vertical exact sequence of the above diagram by $\sO_p$ for a general point $p\in H$. 
Since $p$ does not belong to the Jacobian scheme defined by $\mathcal{J}_{\nabla D_k}$ we have $\mathcal{J}_{\nabla D_k}\otimes \sO_p=\sO_p$ and  $\mathrm{Tor}_1(\mathcal{J}_{\nabla D_k},\sO_p)=0$.
Since $p\in H\subset Z_k$ we have  $\mathrm{rank}(\mathcal{Q}\otimes \sO_p)\ge 2$; it implies $\mathrm{rank}(\mathcal{F}\otimes \sO_p)\ge 3$ in other words that 
$p\in  \mathrm{sg}(\mathcal{F})$;  this contradicts $\mathrm{codim}(\mathrm{sg}(\mathcal{F}), \p^2)\ge 2$, proved in lemma \ref{singF}.

\smallskip

Then $\mathcal{Q}$ is the ideal sheaf of the codimension two scheme $Z_k$, i.e. $ \mathcal{Q}=\mathcal{J}_{Z_k}(n(2-k)-1)$ and we have an 
exact sequence $$ 
\begin{CD}
0@>>> \mathcal{J}_{Z_k}(n(2-k)-1) @>>>  \mathcal{F}  @>>>  \mathcal{J}_{\nabla D_k}(nk-1) @>>>0.
\end{CD}
$$
\end{proof}
From this commutative diagram we obtain the following lemma.
\begin{lem}
\label{sum-of-c2}
 Let  $D_k$ be a union of $k\ge 2$ members of $\mathcal{C}(f,g)$.  Then
 $$c_2(\mathcal{J}_{\nabla D_k})+c_2(\mathcal{J}_{Z_k})=3(n-1)^2+n^2(k-1)^2.$$
\end{lem}
\begin{proof}
According to the above commutative diagram we compute $c_2(\mathcal{F})$ in two different ways.
The horizontal exact sequence gives $c_2(\mathcal{F})=4(n-1)^2$ when the vertical one gives 
$c_2(\mathcal{F})=c_2(\mathcal{J}_{\nabla D_k})+c_2(\mathcal{J}_{Z_k})+(n-1)^2-n^2(k-1)^2.$ The lemma is proved by eliminating $c_2(\mathcal{F})$.
\end{proof}
\subsection{Free divisors in the pencil}
When $D_k$ contains the divisor $D^{\mathrm{sg}}$ of all the singular members of the pencil we show now that, under some supplementary condition on the nature of the singularities, it is free 
with exponents $(2n-2, n(k-2)+1)$.
\begin{thm} \label{thm1}
 Assume that the base locus of the pencil $\mathcal{C}(f,g)$ is smooth, $n\ge 1$ and $k> 1$.  
 Then, 
  $D_k$ is free with exponents $(2n-2,n(k-2)+1)$ if and only if $D_k\supseteq D^{\mathrm{sg}}$ and $J_{\nabla D_k}$ is locally a complete intersection 
  at every $p\in \mathrm{Sing}(\mathcal{C})$.
\end{thm}
\begin{proof}
Let us remark first that $D_k$ is free with exponents $(2n-2,n(k-2)+1)$ if and only if the zero set  $Z_k$ of the ``canonical section'' $s_{\delta,k}$ is empty.
Indeed, if $Z_k=\emptyset$ then $\HH^1(\cT_{D_k}(m))=0$ for all $m\in \Z$ and by 
Horrocks'criterion it implies that $D_k$ is free with exponents $(2n-2,n(k-2)+1)$. The other direction is straitforward.

\smallskip

According to 
lemma 2.6, $Z_k=\emptyset$ if and only if $c_2(\mathcal{J}_{\nabla D_k})=n^2(k-1)^2+3(n-1)^2.$ 
Moreover it is well known (see \cite[section 1.3]{Schenck-Toheaneanu} for instance) that the length of the Jacobian scheme of $D_k$ is 
$c_2(\mathcal{J}_{\nabla D_k})=\sum_{p\in \mathrm{Sing}(D_k)} \tau_p(D_k)$ where $\tau_p(D_k)$ is the Tjurina number of $D_k$ at $p\in D_k$
(this number $\tau_p(D_k)$ is the length of the subscheme of the Jacobian scheme supported by $p$).
Then to prove the theorem we show below
that $\sum_{p\in \mathrm{Sing}(D_k)} \tau_p(D_k)=n^2(k-1)^2+3(n-1)^2$ if and only if $D_k\supseteq D^{\mathrm{sg}}$ and 
$J_{\nabla D_k}$ is locally a complete intersection at every $p\in \mathrm{Sing}(\mathcal{C})$.

\smallskip

%
%
The Jacobian scheme of $D_k$ is supported by the base locus $B$ of the pencil and by the singular points of the $k$ curves forming $D_k$. 
The syzygy $\nabla f\wedge \nabla g$ of $J_{\nabla D_k}$ does not vanish at  $p\in B$; it implies that $J_{\nabla D_k}$ is locally a complete intersection 
at $p\in B$; according to \cite[section 1.3]{Schenck-Toheaneanu} it gives $\tau_p(D_k)=\mu_p(D_k)$, 
 where this last number is the Milnor number of $D_k$ at $p$. Since $p$ is an ordinary singular point of multiplicity $k$ we obtain $\mu_p(D_k)=(k-1)^2$.
Then $\sum_{p\in B}\tau_p(D_k)=n^2(k-1)^2$.

\smallskip

Let us compute now $\sum_{p\in \mathrm{Sing}(D_k)\setminus B}\tau_p(D_k)$. 

Let $C_p\subset D_k$ be the unique curve in the pencil singular at $p\in  \mathrm{Sing}(D_k)\setminus B$. We can verify without difficulties that 
their Jacobian ideals coincide locally at $p\in \mathrm{Sing}(D_k)\setminus B$, in particular $\tau_p(D_k)=\tau_p(C_p)$ and
$\sum_{p\in \mathrm{Sing}(D_k)\setminus B}\tau_p(D_k)=\sum_{p\in \mathrm{Sing}(D_k)\setminus B}\tau_p(C_p).$

Let $I=(\nabla f \wedge \nabla g)$  be the ideal generated by the two by two minors of the $3\times 2$ matrix 
$(\nabla f, \nabla g)$ 
defining the scheme $\mathrm{sg}(\mathcal{F})$. 
Let   $\mathrm{sg}(\mathcal{F})_p$ 
be the subscheme of $\mathrm{sg}(\mathcal{F})$ supported  by the point $p$. 
We have  seen in lemma 2.2  that 
$\mathrm{sg}(\mathcal{F})$ is supported by the whole set of singular points of the pencil and that
$l(\mathrm{sg}(\mathcal{F}))=\sum_{p\in \mathrm{Sing}(\mathcal{C})} l(\mathrm{sg}(\mathcal{F})_p)=3(n-1)^2$. 
Let us consider the situation in a fixed point $p\in \mathrm{Sing}(D_k)\setminus B$. To simplify the notation assume that $f=0$ is an equation for $C_p$. 
Then the other curves of the pencil do not pass through $p$, in particular $g(p)\neq 0$.
Since 
$<\nabla f\wedge \nabla g, \nabla g>=0$, $\nabla g$ is a syzygy of $I$ that do not vanish at $p\in \mathrm{Sing}(D_k)\setminus B$. It implies that 
 $I$ is locally a complete intersection at $p$. Since  the ideal $I_p$  is obtained by taking the two by two minors of the matrix 
$(\nabla f, \nabla g)$ in the local ring $S_p$ the inclusion $I_p \subset J_{\nabla f,p}$ is straitforward;
This inclusion implies  $\tau_p(C_p)\le l(\mathrm{sg}(\mathcal{F})_p)$ because $l(\mathrm{sg}(\mathcal{F})_p)=l(S_p/I_p)$.

\smallskip

Then $Z_k=\emptyset$ if and only if $\sum_{p\in \mathrm{Sing}(D_k)\setminus B}\tau_p(C_p)=\sum_{p\in \mathrm{Sing}(\mathcal{C})}l(\mathrm{sg}(\mathcal{F})_p).$
And this equality is verified if and only if 
$\mathrm{Sing}(D_k)\setminus B=\mathrm{Sing}(\mathcal{C})$ and $l(\mathrm{sg}(\mathcal{F})_p)=\tau_p(C_p)$ for all $p\in \mathrm{Sing}(\mathcal{C})$.
The second equality is equivalent to the equality $I_p = J_{\nabla f,p}$
which implies that the Jacobian ideal of $C_p$ is  locally a complete intersection at $p$. 

\smallskip

Since the Jacobian ideals of $D_k$ and $C_p$ coincide locally at $p\in \mathrm{Sing}(D_k)\setminus B$ this proves that 
$Z_k=\emptyset$ if and only if $D_k$ contains all the singular members of the pencil ($\mathrm{Sing}(D_k)\setminus B=\mathrm{Sing}(\mathcal{C})$) 
and the Jacobian ideal of $D_k$ is locally a complete intersection in every singular point of the pencil ($\mu_p(C_p)=\tau_p(C_p)$ for all $p\in \mathrm{Sing}(\mathcal{C})$).
\end{proof}
\begin{rmk}
We insist on the following fact: if the Jacobian ideal of $D_k$ is not locally a complete intersection at $p$ then $Z_k$ is not empty because $p\in \mathrm{supp}(Z_k)$ even if 
$D_k$ contains all the singular curves.
\end{rmk}
\begin{rmk}
Except for $k=1$ and $(k,n)=(3,n)$ with $n\ge 5$ we can omit to precise the splitting type of $\cT_{D_k}$. Moreprecisely, when $k\neq 1$ and 
$(k,n)\neq (3,n)$ with $n\ge 5$ the following equivalence is true:

\textit{The divisor $D_k$ is free if and only if $D_k\supseteq D^{\mathrm{sg}}$ and the Jacobian ideal of $D_k$ is locally a complete intersection.}

\smallskip

Let us prove it.

\smallskip

$\bullet$ It is not true for $k=1$. Indeed, 
one can choose a free curve $\{f=0\}$ and take any other singular curve $\{g=0\}$ with the same degree
such that the intersection locus $\{f=g=0\}$ is smooth. Then $\{f=0\}$ is free even if it does not pass through all the singular points of the pencil. 

\smallskip

$\bullet$  When $k=2$ it is true. Under the hypothesis of smoothness of the base locus of the pencil, there is only one example of union of $k=2$ curves that is free: 
the union of two lines (i.e. $k=2$ and $n=1$).
Indeed, assume that $k=2$ and that $D_2$ is free. Then we have for $0 \le t \le n-1$:
 $$
 \begin{CD}
  0@>>> \sO_{\p^2}(2-2n)  @>s_{\delta,2}>>\sO_{\p^2}(2-2n+t)\oplus \sO_{\p^2}(-1-t)@>>> \mathcal{J}_{Z_2}(-1) @>>>0.
 \end{CD}
$$ 
The Jacobian ideal of $D_2$ defines a finite scheme of length smaller or equal to $2(n-1)^2+n^2$ since for one reduced 
curve of degree $n$ the maximal length of its Jacobian scheme is $(n-1)^2$, attained for $n$ concurrent lines. Then, according  to lemma \ref{sum-of-c2}, 
the length  of the zero scheme $Z_2$ is at  least $(n-1)^2$. 
When $n=1$ we have $\cT_{D_2}=\sO_{\p^2}\oplus \sO_{\p^2}(-1)$ and $Z_2=\emptyset$. Otherwise, when $n>1$, $Z_2$ cannot be empty. This implies that $t> 0$ when $n>1$.
The zero scheme $Z_2$ is then a complete intersection
$(t,2n-3-t)$; its length is $t(2n-3-t)$ and this number should be greater than $(n-1)^2$. This never happens for $n>1$. We have shown that the union of two curves of degree $n>1$
meeting along $n^2$ distinct points is never free. If we accept a non smooth base locus then it is possible. Consider for instance an arrangement of four lines with 
a triple point. This arrangement is free
with exponents $(1,2)$ and it can be seen as the union of the (only) two singular conics of a pencil of tangent conics (at the triple point). 

\smallskip

$\bullet$  When $k=3$ it is true for $n< 5$ but it is still an open question for $n\ge 5$. Indeed we have for $k=3$
 $$
 \begin{CD}
  0@>>> \sO_{\p^2}  @>s_{\delta,3}>>\cT_{D_3}(2n-2)@>>> \mathcal{J}_{Z_3}(n-3) @>>>0.
 \end{CD}
$$ 
If $n\le 4$,  $D_3$ is free and there is  only one possible splitting for $\cT_{D_3}(4)$, that is  $\sO_{\p^2}\oplus \sO_{\p^2}(-1)$ for $n=2$, $\sO_{\p^2}^2$ for $n=3$ and 
$\sO_{\p^2}\oplus \sO_{\p^2}(1)$ for $n=4$. 
It implies necessarly that $Z_3=\emptyset$. Then the Jacobian scheme of the union of the three cubics has exactly length $12+9\times 4$ (resp. $27+16\times 4$). In other words the sum of  the Tjurina numbers of each cubic (resp. quartic) is $4$
(resp. is $9$). This is possible if and only if 
each cubic (resp. quartic) is a product of three (resp. four) concurrent lines. 
The pencil is equivalent to $(x^3-y^3,y^3-z^3)$ (resp. $(x^4-y^4,y^4-z^4)$) and it contains also the cubic form $x^3-z^3$ (resp. $x^4-z^4$).

\smallskip

More generally, if $k=3$ and $Z_3=\emptyset$ then $\HH^1(\cT_{D_3}(m))=0$ for all $m\in \Z$ and by Horrocks' criterion it implies $\cT_{D_3}(2n-2)=\sO_{\p^2}\oplus \sO_{\p^2}(n-3)$. 
This arises if and only if the pencil is projectively equivalent to
$(x^n-y^n, y^n-z^n)$.

\smallskip

But if $n\ge 5$ one  could possibly have $D_3$ free but $Z_3\neq \emptyset$; i.e. $\cT_{D_3}(2n-2)= \sO_{\p^2}(t)\oplus \sO_{\p^2}(n-3-t)$ with $1\le t \le \frac{n-3}{2}$ and $Z_3$ is a non empty 
 complete intersection $t(n-3-t)$.
In other words $Z_3$ is supported by the singularities where $D_3$ is not a local complete intersection and by the singularities that do not belong $D_3$. 

\smallskip

$\bullet$  When $k\ge 4$ it is true. The inequality  $k\ge 4$ gives $n(k-2)+1>2n-2$ for all $n$ and there is only one possible splitting
for $\cT_{D_k}$ that is $\sO_{\p^2}(2-2n)\oplus \sO_{\p^2}(n(2-k)-1)$. 
Then   $D_k$ is free if and only if $Z_k=\emptyset$.
\end{rmk}
\subsection{Singular members ommitted}
When $D_k\supset D^{\mathrm{sg}}$ and its Jacobian is locally a complete intersection we have seen in theorem \ref{thm1} that $Z_k=\emptyset$ 
by computing the length of the scheme defined by the Jacobian ideal of $D_k$.
More generally we can describe, at least when the base locus is smooth, the scheme $Z_k$ 
for any union of curves of the pencil.
\begin{thm}\label{thm0}Assume that the base locus of the pencil $\mathcal{C}(f,g)$ is smooth and that the Jacobian ideal of $D^{sg}$
is locally a complete intersection.
Assume also that $D_k$ contains all the singular members of the pencil except the singular curves $C_{\alpha_i,\beta_i}$  for $i=1,\ldots, r$.
Then, $$ \mathcal{J}_{Z_k}=\mathcal{J}_{\nabla C_{\alpha_1,\beta_1}}\otimes \cdots \otimes \mathcal{J}_{\nabla C_{\alpha_r,\beta_r}}.$$
\end{thm}
\begin{rmk}
 When $r=0$ we obtain  the freeness again.
\end{rmk}

\begin{proof}
Since the set of singular points of two distinct curves are disjoint, it is enough to prove it for $r=1$. We recall  that $\mathcal{E}xt^1(\mathcal{J}_{Z},\sO_{\p^2})=\omega_Z$ 
where $Z$ is a finite scheme and $\omega_Z$ is its dualizing sheaf (see \cite[Chapter III, section 7]{Ha}); 
When the finite scheme $Z$ is locally a complete intersection, $\omega_Z=\sO_Z$.
 
\smallskip

 The dual exact sequence of 
$$ 
\begin{CD}
 0@>>> \mathcal{J}_{Z_k}(n(2-k)-1)  @>>>\mathcal{F} @>>> \mathcal{J}_{\nabla D_k}(nk-1)  @>>>0
\end{CD}
$$
is the long exact sequence
$$ 
\begin{CD}
 0@>>> \sO_{\p^2}(1-nk) @>>>\sO_{\p^2}(1-n)^2  @>>> \sO_{\p^2}(n(k-2)+1) @>>> 
\end{CD}
$$
$$ 
\begin{CD}
 @>>> \omega_{\nabla D_k}    @>>> \sO_{\mathrm{sg}(\mathcal{F})} @>>>\omega_{Z_k}  @>>> 0.
\end{CD}
$$
Since $Z_k$ is the zero set of a rank two vector bundle it is locally a complete intersection and $\omega_{Z_k}=\sO_{Z_k}$.
The map $
\begin{CD}
\mathcal{F} @>>> \mathcal{J}_{\nabla D_k}(nk-1)  
\end{CD}
$
can be described by composition; indeed it is given by two polynomials $(U,V)$ such that 
$$ (U,V).(\nabla f,\nabla g)=\nabla (\prod_i (\alpha_i f+\beta_i g)).$$ We find,
$U=\sum_i \alpha_i \prod_{j\neq i}(\alpha_{j}f+\beta_{j} g)$ and $V=\sum_i \beta_i \prod_{j\neq i}(\alpha_{j}f+\beta_{j} g).$
The complete intersection  $T= \{U=0\}\cap \{V=0\}$  defines a multiple structure of length $n^2(k-1)^2$ along
 the base locus. It implies that the scheme defined by the Jacobian ideal $\mathcal{J}_{\nabla D_k}$ contains the scheme $T$. In other words
$\omega_{\nabla D_k}= \sO_T \oplus \frak{R}$ where $\frak{R}$ is supported by the singularities of $D_k$ outside the base locus of the pencil.
Then we have 
$$ 
\begin{CD}
 0@>>> \frak{R}   @>>> \sO_{\mathrm{sg}(\mathcal{F})} @>>>\sO_{Z_k}  @>>> 0.
\end{CD}
$$
The hypothesis 
oon the nature of the singularities implies that  $\frak{R}=\oplus_{i=2,\ldots,s}\sO_{Z_{\alpha_i,\beta_i}}$ and 
$\sO_{\mathrm{sg}(\mathcal{F})}=\oplus_{i=1,\ldots,s}\sO_{Z_{\alpha_i,\beta_i}}$.
 This proves $\omega_{Z_k}=\sO_{Z_{\alpha_1,\beta_1}}$.

\end{proof}
There are exact sequences relating the vector bundles $\cT_{D_k}$ and $\cT_{D_k\setminus C}$ when $C\subset D_k$.
\begin{prop}  \label{moins1} We assume that the base locus of the pencil $\mathcal{C}(f,g)$ is smooth, 
that the Jacobian ideal of $D^{sg}$
is locally a complete intersection and that 
  $D_k$ contains $D^{\mathrm{sg}}$. 
 Let $C$ be a singular member in $\mathcal{C}(f,g)$ and $Z$ its scheme of singular points. Then there is an exact sequence 
  $$ 
\begin{CD}
 0@>>> \cT_{D_k} @>>>\cT_{D_k\setminus C} @>>>\mathcal{J}_{Z/C}(n(3-k)-1) @>>>0,
\end{CD}
$$
where $\mathcal{J}_{Z/C}\subset \sO_C$  defines $Z$ into $C$.
\end{prop}
\begin{proof}
The derivation $(\nabla f\wedge \nabla g).\nabla$ is tangent to $D_k$ then also to $D_k\setminus C$. It induces  the following 
commutative diagram which proves the proposition:
  $$ 
\begin{CD}
@.   @. 0 @. 0 @.\\
@.   @. @VVV  @VVV\\
0@>>> \sO_{\p^2}(2-2n)  @>>>\cT_{D_k} @>>> \sO_{\p^2}(n(2-k)-1) @>>>0\\
@. @| @VVV @VVV   @.\\
 0@>>> \sO_{\p^2}(2-2n)  @>>>\cT_{D_k\setminus C}  @>>> \mathcal{J}_{Z}(n(3-k)-1) @>>>0\\
@.   @. @VVV  @VVV\\
@.   @. \mathcal{J}_{Z/C}(n(3-k)-1)@=  \mathcal{J}_{Z/C}(n(3-k)-1) @.\\
@.   @. @VVV  @VVV\\
@.   @. 0 @. 0 @.
\end{CD}
$$
\end{proof}

\section{The pencil contains a non-reduced curve}

When the pencil $\mathcal{C}(f,g)$ contains a non-reduced curve, the arguments used in the previous sections are not valid since the scheme defined by the Jacobian ideal contains a divisor.
We  have to remove this divisor somehow. Remember that if two curves of the pencil are multiple then the general curve is singular. 
So let us consider that there is only one curve that is not reduced. Let $hh_1^{r_1}\cdots h_s^{r_s}=0$ be the equation of this unique non-reduced curve where $h=0$ is reduced, 
$\mathrm{deg}(h_i)=m_i\ge 1$ and $r_i\ge 2$. Since the derivation $\frac{1}{\prod_ih_i^{r_i}}(\nabla f\wedge \nabla g).\nabla$ is still
tangent to all curves of the pencil, we believe that the following statement is true:

\begin{cnj*}\label{cnj1}
 Let $hh_1^{r_1}\cdots h_s^{r_s}=0$ be the equation of the unique non-reduced curve where $\{h=0\}$ is reduced, 
$\mathrm{deg}(h_i)=m_i\ge 1$ and $r_i\ge 2$.

\smallskip

Then,  $D_k$ is free with exponents $(2n-2-\sum_{i=1}^{i=s}(r_i-1)m_i,n(k-2)+1)$ if and only if $D_k\supseteq D^{\mathrm{sg}}$ and the Jacobian ideal of
the reduced structure of $D_k$ is locally a complete intersection.
\end{cnj*}

\section{Examples}

Let us call $\Sigma_3\subset \p^9=\p(\HH^0(\sO_{\p^2}(3)))$ the hypersurface of singular cubics.
It is well known that its degree is $12$ (see \cite{GKZ}, for instance).
\begin{itemize}
 \item \textbf{Pappus arrangement freed by nodal cubics}:  Let us consider  the divisor of the nine lines appearing in the Pappus arrangement; this divisor is the union of three triangles $T_1,T_2,T_3$  with nine base points.
The pencil generated by $T_1$ and $T_2$ contains $3$ triangles (each one represents a triple point 
in $\Sigma_3$); since $9<12$,  singular cubics are missing in the pencil. There is no other triangle and no smoth conic+line in the Pappus pencil, when it is general enough.
We can conclude that the missing cubics are, in general, nodal cubics $C_1,C_2,C_3$.

Let  $D=T_1\cup T_2 \cup T_3\cup C_1\cup C_2\cup C_3$ be the union of all singular fibers in the pencil generated by $T_1$ and $T_2$.
Then, according to theorem \ref{thm1} we have   $$\cT_{D}=\sO_{\p^2}(-4)\oplus \sO_{\p^2}(-13).$$
\item  \textbf{Pappus arrangement}: Let    $T_1\cup T_2\cup T_3$ be the divisor  consisting of the nine lines of the projective Pappus arrangement and 
  $D=T_1\cup T_2\cup T_3\cup C_1\cup C_2\cup C_3$ be the union of all singular fibers in the pencil generated by two triangles among the $T_i$'s. Let us call $K:=C_1\cup C_2\cup C_3$ the union of the 
nodal cubics and $z_1,z_2,z_3$ their nodes.
Then we have, according to theorem \ref{thm0},   $\cT_{D} =\sO_{\p^2}(-4)\oplus \sO_{\p^2}(-13)$ and an exact sequence 
  $$ 
\begin{CD}
 0@>>> \sO_{\p^2}(-4)  @>>>\cT_{D\setminus K}  @>>> \mathcal{I}_{z_1,z_2,z_3}(-4) @>>>0.
\end{CD}
$$
The logarithmic bundle $\cT_{D\setminus K}$ associated to the Pappus configuration is semi-stable and its divisor of jumping lines is the triangle
$z_1^{\vee}\cup z_2^{\vee}\cup z_3^{\vee}$ as it is proved  by retricting the above exact sequence to any line through one of the zeroes.
\begin{figure}[h!]
  \label{figure-pappus}
\centering
\begin{tikzpicture}[scale=1.3]
\draw (-1,-2) -- (-1,2.2);
\draw (1,-2) -- (1,2);

\draw (-1,-1) node {$\bullet$};
\draw (-1,0) node {$\bullet$};
\draw (-1,2) node {$\bullet$};
\draw (1,-1) node {$\bullet$};
\draw (1,0) node {$\bullet$};
\draw (1,1) node {$\bullet$};
\draw (0,-0.5) node {$\bullet$};
\draw (0.2,0.2) node {$\bullet$};
\draw (0.33,0.66) node {$\bullet$};
\draw (-0.2,-1.2) -- (0.5,1.25);
\draw (-1,-1) -- (1,0);
\draw (-1,-1) -- (1,1);
\draw (-1,0) -- (1,-1);
\draw (-1,0) -- (1,1);
\draw (-1,2) -- (1,0);
\draw (-1,2) -- (1,-1);

\end{tikzpicture}
  \caption{Pappus arrangement}
\end{figure}
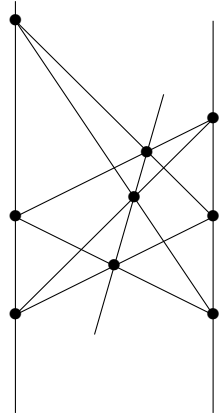

\item \textbf{Hesse arrangement}: Let us consider the pencil  generated by a smooth cubic  $C$ and its hessian $\mathrm{Hess}(C)$.
The pencil contains $4$ triangles $T_1, T_2,T_3,T_4$ and since the degree of $\Sigma_3$ is $12$, no other singular cubic can be present. Let us call $D$ the union of these four triangles.
Then, according to theorem \ref{thm1} we have   $$\cT_{D}=\sO_{\p^2}(-4)\oplus \sO_{\p^2}(-7).$$
\end{itemize}

\end{document}